\documentclass[11pt]{amsproc}

\usepackage{amsmath}
\usepackage{amsthm}
\usepackage{amssymb}
\usepackage{mathtools}
\usepackage{mathrsfs}
\usepackage{bm}
\usepackage{tikz}
\usepackage{tikz-cd}
\usepackage[most]{tcolorbox}
\usepackage{graphicx}
\usepackage{xcolor}
\usepackage{fullpage}
\usepackage[hypertexnames=false]{hyperref}
\usepackage{xurl}
\usepackage{thmtools}
\usepackage[capitalize, nameinlink]{cleveref}
\usepackage{float}
\usepackage{datetime}
\usepackage{aligned-overset}
\usepackage{ulem}
\usepackage{enumitem}
\usepackage{url}
\usepackage{physics}
\usepackage{stmaryrd}
\usepackage{listings}
\usepackage[abbrev]{amsrefs}

\newtheorem{theoremcounter}{Theorem Counter}[section]

\theoremstyle{remark}
\newtheorem{remark}[theoremcounter]{Remark}

\theoremstyle{definition}
\newtheorem{definition}[theoremcounter]{Definition}

\theoremstyle{plain}
\newtheorem{lemma}[theoremcounter]{Lemma}
\newtheorem{proposition}[theoremcounter]{Proposition}

\newtheorem{theorem}[theoremcounter]{Theorem}

\numberwithin{equation}{section}

\newcommand{\bbC}{\mathbb{C}}
\newcommand{\bbN}{\mathbb{N}}
\newcommand{\bbZ}{\mathbb{Z}}

\newcommand{\bbQ}{\mathbb{Q}}

\newcommand{\calA}{\mathcal{A}}

\newcommand{\stirlingone}[2]{\genfrac{[}{]}{0pt}{}{#1}{#2}}

\allowdisplaybreaks[4]
\DeclareMathOperator{\den}{den}

\newcommand{\stirlingII}[2]{\genfrac{\{}{\}}{0pt}{}{#1}{#2}}

\begin{document}

\title[]{On finite analogues of Dobi\'{n}ski's formula and of Euler's constant via Gregory polynomials}

\author[]{Toshiki Matsusaka}
\address{Faculty of Mathematics, Kyushu University, Motooka 744, Nishi-ku, Fukuoka 819-0395, Japan}
\email{matsusaka@math.kyushu-u.ac.jp}

\author[]{Taichi Miyazaki}
\address{Department of Mathematics, Kyushu University, Motooka 744, Nishi-ku, Fukuoka 819-0395, Japan}
\email{miyazaki.taichi.478@s.kyushu-u.ac.jp}

\author[]{Shunta Yara}
\address{Department of Mathematics, Kyushu University, Motooka 744, Nishi-ku, Fukuoka 819-0395, Japan}
\email{yara.shunta.444@s.kyushu-u.ac.jp}


\maketitle

\begin{abstract}
	We study a finite analogue of Dobi\'{n}ski's formula, which is related to the Napier constant $e$, and its Bessel-type generalizations. Furthermore, using Gregory polynomials, we extend the results of Kaneko--Matsusaka--Seki on finite analogues of Euler's constant, and compare them with the Wilson-type analogue $\gamma_\calA^\mathrm{W}$.
\end{abstract}

\section{Introduction}

Let $\calA$ be the quotient ring defined by
\[
	\calA\coloneqq \bigg(\prod_{p:\text{ prime}}\bbZ/p\bbZ\bigg)\Big/\bigg(\bigoplus_{p:\text{ prime}}\bbZ/p\bbZ \bigg).
\]
The diagonal embedding of the rational field $\mathbb{Q}$ into $\mathcal{A}$ endows $\calA$ with a natural structure of a $\bbQ$-algebra. For a rational number, its $p$-component is not well-defined for the (finitely many) primes $p$ dividing its denominator. In such cases, we assign an arbitrary value (for instance, $0$) to the $p$-component. Kaneko and Zagier~\cite{KanekoZagier} introduced \textit{finite multiple zeta values} as analogues of multiple zeta values in $\calA$. For a tuple of positive integers $(k_1, \dots, k_r)$, they are defined by
\[
	\zeta_\calA(k_1, \dots, k_r) \coloneqq \left(\sum_{0 < m_1 < \cdots < m_r < p} \frac{1}{m_1^{k_1} \cdots m_r^{k_r}} \bmod{p} \right)_p \in \calA.
\]
These values have been systematically studied from the viewpoint of their analogy with multiple zeta values, and many structural similarities between the two theories have been revealed.

Moreover, beyond finite multiple zeta values, there have been various attempts to construct $\calA$-analogues of arithmetically interesting quantities, leading to what may be called an emerging ``arithmetic of $\calA$". Examples include $\calA$-analogues of algebraic numbers~\cites{Rosen2020, RTTY2024}, constructions of transcendental numbers~\cites{AnzawaFunakura2024, LucaZudilin2025, LucaZudilin2025-pre}, and an $\calA$-analogue of the Euler constant~\cite{KMS2025}. On the other hand, even proving that such elements are nonzero is in general a difficult problem. In fact, at present, it is not known whether there exists a tuple $(k_1, \dots, k_r)$ such that $\zeta_\calA(k_1, \dots, k_r) \neq 0$. In~\cite{KanekoZagier}, the element defined using \textit{Fermat quotients},
\begin{align}\label{def:log-A}
	\log_\calA(x) \coloneqq \left(\frac{x^{p-1} -1}{p} \bmod{p}\right)_p \in \calA \quad (x \in \bbQ^\times)
\end{align}
is regarded as an $\calA$-analogue of the logarithm, since it satisfies the functional equation $\log_\calA(xy) = \log_\calA(x) + \log_\calA(y)$. In this case, assuming the ABC conjecture, one can show that $\log_\calA(x) \neq 0$ for $x \in \bbQ^\times \setminus \{-1, 1\}$ (due to Silverman~\cite{Silverman1988}).

The main objectives of this article are twofold. The first is to establish an $\calA$-analogue of Dobi\'{n}ski's formula related to the Napier constant $e$, thereby providing a candidate for an $\calA$-analogue of $e$. We further extend this result to a more general Bessel-type setting. The second is to extend the $\calA$-analogues of the Euler constant studied by the first author together with Kaneko and Seki~\cite{KMS2025}, using Gregory polynomials.

\textit{Dobi\'{n}ski's formula} is the identity
\[
	D(n)\coloneqq \sum_{k = 0}^\infty\frac{k^n}{k!} = b(n)e,
\]
which was considered in his paper (1877)~\cite{dobinski1877summirung}. Here, $(b(n))_{\ge 0}$ is the sequence of integers defined by $b(0) = 1$ together with the recurrence relation
\[
	b(n + 1) = \sum_{k = 0}^n\binom{n}{k}b(k) \quad (n\ge 0).
\]
This yields $(b(n))_{n\ge 0} = (1,1,2,5,15,52,203,877,\ldots)$, which is known as \textit{Bell numbers}, named after the work of Bell~\cite{Bell1938} (OEIS~\cite[\href{https://oeis.org/A000110}{A000110}]{OEIS}). Motivated by Dobi\'nski's formula, we define 
\[
	D_{\calA}(n) 
	\coloneqq
	\left(\sum_{k=0}^{p-1} \frac{k^n}{k!} \bmod p \right)_p
	\in \calA.
\]
and introduce a finite analogue of the constant $e$ by setting $e_\calA \coloneqq D_\calA(0)$. Then, the following $\calA$-analogue of Dobi\'{n}ski's formula holds.

\begin{theorem}\label{analogue of Dobinski}
	For any integer $n\ge 0$, we have
	\[
		D_\calA(n) = b(n)e_\calA + g(n),
	\]
	where the sequence $(g(n))_{n\ge 0} = (0,1,1,3,9,31,121,523,\ldots)$ is given by $g(0) = 0, g(1) = 1$, and by the same recurrence relation as the Bell numbers, that is, 
	\[
		g(n + 1) = \sum_{k = 0}^n\binom{n}{k}g(k)\quad (n\ge 1),
	\]
  $($see also OEIS~\cite[\href{https://oeis.org/A040027}{A040027}]{OEIS}$)$.
\end{theorem}

Next, we further generalize Dobi\'nski's formula by introducing
\[
	D_r(n)\coloneqq \sum_{k = 0}^\infty\frac{k^n}{(k!)^r}
\]
for $r \ge 1$. By definition, $D_1(n) = D(n)$. Moreover, since $D_2(0) = I_0(2)$, where $I_0$ denotes the $I$-Bessel function, this may be regarded as a Bessel-type generalization. Numerical experiments suggest the following. For $r=2$, one has
\[
	D_2(n) = b_{2,0}(n) D_2(0) + b_{2,1}(n) D_2(1).
\]
The coefficients are given as follows:
\[
	\begin{array}{c|ccccccccc}
		n & 0 & 1 & 2 & 3 & 4 & 5 & 6 & 7 & 8 \\ \hline
		b_{2,0}(n) & 1 & 0 & 1 & 1 & 2 & 5 & 13 & 36 & 109 \\ \hline
		b_{2,1}(n) & 0 & 1 & 0 & 1 & 2 & 4 & 10 & 29 & 90
	\end{array}
\]
This observation appears in the OEIS~\cite[\href{https://oeis.org/A086880}{A086880} and \href{https://oeis.org/A246118}{A246118}]{OEIS}. 
In general, we prove the following.

\begin{theorem}\label{Dobinski with r}
	For any integers $r\ge 1$ and $n\ge 0$, we have
	\begin{align}\label{eq:Thm12-first}
		D_r(n) = \sum_{j = 0}^{r - 1}b_{r,j}(n)D_r(j),
	\end{align}
	where the coefficients $b_{r,j}(n)$ are defined by $b_{r,j}(n) = \delta_{j,n}$ for $0\le n\le r - 1$ and the recurrence relation
	\begin{align}\label{def of brj}
		b_{r,j}(n + r) = \sum_{k = 0}^n\binom{n}{k}b_{r,j}(k)\quad (n\ge 0).
	\end{align}
	In addition, $D_r(0),\ldots, D_r(r - 1)$ are linearly independent over $\overline{\bbQ}$.
\end{theorem}

Furthermore, we introduce an $\calA$-analogue of \cref{Dobinski with r}. For any integer $r \ge 1$, we define $D_{r,\calA}(n)$ by
\begin{align}\label{def:DrA}
	D_{r,\calA}(n)
	\coloneqq
	\left(\sum_{k=0}^{p-1} \frac{k^n}{(k!)^r} \bmod p\right)_{p}
	\in \calA.
\end{align}
Then the following theorem holds. In particular, \cref{analogue of Dobinski} corresponds to the special case $r=1$ of this theorem.

\begin{theorem}\label{analogue of Dobinski with r}
	For any integers $r \ge 1$ and $n\ge 0$, we have  
	\begin{align}
		D_{r,\calA}(n) = \sum_{j = 0}^{r - 1}b_{r,j}(n)D_{r,\calA}(j) + g_r(n),\label{DrA}
	\end{align}
 	where the sequence $(g_r(n))_{n \ge 0}$ is given by the same recurrence relation
	\[
		g_r(n + r) = \sum_{k = 0}^n\binom{n}{k}g_r(k)\quad (n\ge 1)
	\]
	as \eqref{def of brj}, with the initial values $g_r(n) = (-1)^{r - 1}\delta_{n,r}$ for $0\le n\le r$.
\end{theorem}

In \cref{Dobinski}, we prove the above theorems in more general setting. However, it remains unknown whether $e_\calA$ is irrational in the ring $\calA$, or whether $D_{r,\calA}(j)$ for $0 \le j \le r-1$ are $\bbQ$-linearly independent. As noted in OEIS~\cite[\href{https://oeis.org/A064384}{A064384}]{OEIS}, only finitely many, and in fact very few, primes are currently known for which the $p$-component of $e_\calA$ vanishes. This situation is very similar to that of Wolstenholme primes \cite[\href{https://oeis.org/A088164}{A088164}]{OEIS}, which are related to finite multiple zeta values, and Wilson primes \cite[\href{https://oeis.org/A007540}{A007540}]{OEIS}, which are connected to the finite Euler constant discussed below.

The second topic concerns $\calA$-analogues of the Euler constant. In~\cite{KMS2025}, the authors focused on the classical formulas for Euler's constant due to Mascheroni and its generalization by Kluyver, and introduce corresponding $\calA$-analogues. Here, Mascheroni's formula is given by
\[
	\gamma = \sum_{n=1}^\infty \frac{(-1)^{n-1} G_n}{n}	
\]
in terms of the \textit{Gregory coefficients} defined by
\begin{align}\label{def:Gregory-coeff}
	\frac{t}{\log(1+t)} = \sum_{n=0}^\infty G_n t^n,
\end{align}
(see OEIS~\cite[\href{https://oeis.org/A002206}{A002206}]{OEIS}). From these formulas, finite analogues of Euler's constant, $\gamma_\calA^\mathrm{M}$ and $\gamma_\calA^{\mathrm{K},m}$ ($m \ge 1$), are constructed in~\cite{KMS2025}, and compared with another finite analogue defined via the Wilson quotient
\begin{align}\label{def:Wilson-gamma}
	\gamma_\calA^\mathrm{W} \coloneqq \left(\frac{(p-1)!+1}{p} \bmod{p}\right)_p \in \calA.
\end{align}
More precisely, as stated in~\cite[Remark 4.2]{KMS2025}, they showed that ``their various analogues of $\gamma$ in $\calA$ differ only by linear combinations of $\log_\calA(j)$ ($j \in \bbN$) and $1$". 
In this article, we generalize the Gregory coefficients to Gregory polynomials, thereby extending the results of \cite{KMS2025} and providing a more concise proof. More detailed discussions will be presented in \cref{KMSw/parameters}.

\section*{Acknowledgements}

The authors would like to thank Prof.~Masanobu Kaneko for suggesting Dobi\'{n}ski's formula and its generalizations. The first author was supported by JSPS KAKENHI (JP21K18141 and JP24K16901).

\section{\texorpdfstring{Generalized Dobi\'nski's formula and its $\calA$-analogue}{Generalized Dobi\'nski's formula and its $\calA$-analogue}}\label{Dobinski}

In this section, we prove \cref{Dobinski with r} and \cref{analogue of Dobinski with r} in a more general setting.

\subsection{$\overline{\bbQ}$-linear independence}

By applying the refined version of the Siegel--Shidlovskii theorem due to Beukers~\cite{Beukers2006}, we prove the following theorem as a generalization of the latter assertion of \cref{Dobinski with r}, namely, that the $D_r(j)$'s are linearly independent over $\overline{\bbQ}$. 

\begin{definition}\label{def:Dr-x}
	For any integers $r \ge 1$ and $n \ge 0$, we define
	\[
		D_r(n; x) \coloneqq \sum_{k=0}^\infty \frac{k^n x^k}{(k!)^r}.
	\]
\end{definition}

This provides a generalization of the functions $e^x = D_1(0; x)$, $I_0(2\sqrt{x}) = D_2(0;x)$, and so on.

\begin{theorem}\label{thm:Drx-Q-indep}
	For any $x \in \overline{\bbQ} \setminus \{0\}$, the values $D_r(0; x), D_r(1; x), \dots, D_r(r-1; x)$ are linearly independent over $\overline{\bbQ}$. In particular, when $x=1$, we obtain the latter part of \cref{Dobinski with r}.
\end{theorem}

We first recall the notion of $E$-functions.

\begin{definition}
	An entire function $f$ given by a power series
	\[
		f(z) = \sum_{n=0}^\infty a_n \frac{z^n}{n!} \quad (a_n \in \overline{\mathbb{Q}})
	\]
	is called an \textit{$E$-function} if it satisfies both of the following conditions.
	\begin{enumerate}[label=(\arabic*)]
		\item The function $f$ satisfies a linear differential equation with coefficients in $\overline{\mathbb{Q}} [z]$.
		\item For the sequence of coefficients $(a_n)_{n \ge 0}$,
		\[
			\log \overline{|a_n|} = O(n), \quad \log (\den (a_0, \ldots, a_n)) = O(n)
		\]
		hold. Here, $\overline{|a_n|}$ is the maximum absolute value of the conjugates of $a_n \in \overline{\mathbb{Q}}$ (the roots of the minimal polynomial of $a_n$), and $\den(a_0, \ldots, a_n)$ is the least common multiple of the denominators of $a_0, \ldots, a_n$ (the smallest integer $d \ge 1$ such that $d a_k$ is an algebraic integer).
	\end{enumerate}
\end{definition}

The exponential function is the most typical example of $E$-functions. For a collection of $E$-functions, the following theorem holds.

\begin{theorem}[{Refined Siegel--Shidlovskii theorem~\protect{\cite[Corolary 1.4]{Beukers2006}}}]\label{thm:ref-SS}
	Let $f_1, \ldots, f_n$ be a set of $E$-functions which satisfy the system of first order equations
	\[
		\frac{\dd}{\dd z}
		\begin{pmatrix}
			y_1 \\ \vdots \\ y_n
		\end{pmatrix} = A \begin{pmatrix}
		y_1 \\ \vdots \\ y_n
		\end{pmatrix}
	\]
	where $A$ is an $n \times n$-matrix with entries in $\overline{\mathbb{Q}} (z)$. Denote the common denominator of the entries of $A$ by $T(z)$. Suppose that $f_1 (z), \ldots, f_n (z)$ are linearly independent over $\overline{\bbQ} (z)$. Then for any $\xi \in \overline{\bbQ}$, with $\xi T(\xi) \ne 0$, the numbers $f_1 (\xi), \ldots, f_n (\xi)$ are $\overline{\mathbb{Q}}$-linear independent. 
\end{theorem}


This provides, for instance, a generalization of the Lindemann--Weierstrass theorem. To prove \cref{thm:Drx-Q-indep} using this, it suffices to find a collection of $E$-functions that take $D_r(j; x)$ as special values.

\begin{lemma}
  For a positive integer $r$, the entire function
  \[
    E(z) := \sum_{n=0}^{\infty} \frac{z^{rn}}{(n!)^r}
  \]
  is an $E$-function.
\end{lemma}

\begin{proof}
	Let $\vartheta \coloneqq z(\dd/\dd z)$ be the \textit{Euler operator}. A direct calculation yields
	\[
		\vartheta^r E(z) = (rz)^r E(z),
	\]
	which implies that $E(z)$ satisfies a linear differential equation with coefficients in $\overline{\bbQ}[z]$.
  
  The function $E(z)$ can be written in the form
  \[
    E(z) = \sum_{N=0}^{\infty} a_N \frac{z^N}{N!},
  \]
  where $(a_N)_{N \geq 0}$ is the sequence of integers defined by
  \[
    a_N :=
    \begin{cases}
      (rn)! / (n!)^r & \text{if $N = rn$ for some $n \ge 0$}, \\
      0 & \text{otherwise.}
    \end{cases}
  \]
  In particular, we have $\mathrm{den}(a_0, \dots, a_n) = 1$ and $\overline{|a_n|} = a_n$. It therefore suffices to examine the behavior of $\log a_n$.
  For $N = rn$, by the multinomial theorem,
  \[
    a_N = \binom{rn}{n, \ldots, n}
    \le \sum_{k_1 + \cdots + k_r = rn} \binom{rn}{k_1, \ldots, k_r}
    = r^{rn} = r^N,
  \]
  and therefore $\log a_N = O(N)$. It follows that $E(z)$ is an $E$-function.
\end{proof}

Since $E$-functions are closed under the action of the Euler operator $\vartheta$, defining $E_j(z) \coloneqq \vartheta^j E(z)$ yields a collection of $E$-functions $E_0 (z), E_1 (z), \ldots, E_{r-1} (z)$. 

\begin{lemma}
	The $E$-functions $E_0 (z), E_1 (z), \ldots, E_{r-1} (z)$ are linearly independent over $\mathbb{C} (z)$.
\end{lemma}

\begin{proof}
	We prove this by contradiction. Suppose that there exist polynomials $p_j(z) \in \bbC[z]$ such that
	\begin{align}\label{eq:Contra}
		\sum_{j=0}^{r-1} p_j(z) E_j(z) = 0.
	\end{align}
	Write each $p_j(z) = \sum_{m \ge 0} a_{j,m} z^m$ and assume that $a_{j,m} \neq 0$ for some index. Set $d \coloneqq \max \{\deg p_j : 0 \le j \le r-1\}$. Then we have
	\[
		\sum_{j=0}^{r-1} p_j(z) E_j(z) = \sum_{m=0}^d \sum_{n=0}^\infty Q_m(rn) \frac{z^{rn+m}}{(n!)^r}.
	\]
	where
	\[
		Q_m(x) \coloneqq \sum_{j=0}^{r-1} a_{j,m} x^j.
	\]
	Let $m^* \coloneqq \max \{0 \le m \le d : Q_m(z) \not\equiv 0\}$ and write $m^* = l^* r+m_0$ with $l^* \ge 0$ and $0 \le m_0 \le r-1$. For any $N \ge l^*$, the coefficient of $z^{rN+m_0}$ is given by
	\begin{align*}
		\sum_{m=0}^d \sum_{\substack{n \ge 0 \\ rn + m = rN+m_0}} \frac{Q_m(rn)}{(n!)^r} &= \sum_{l=0}^{l^*} \frac{Q_{lr+m_0}(r(N-l))}{((N-l)!)^r}\\
			&= \frac{1}{(N!)^r} \sum_{l=0}^{l^*} Q_{lr+m_0}(r(N-l)) \bigg(N(N-1) \cdots (N-l+1)\bigg)^r,
	\end{align*}
	which is equal to $0$ by \eqref{eq:Contra}. Hence, the polynomial
	\[
		\sum_{l=0}^{l^*} Q_{lr+m_0}(r(x-l)) \bigg(x(x-1) \cdots (x-l+1)\bigg)^r
	\]
	vanishes at all integers $x = N \ge l^*$, and therefore must be identically zero. However, the term corresponding to $l = l^*$ has degree at least $l^*r$, while each of the  remaining terms has degree at most $l^* r -1$. Thus, the above polynomial can not be identically zero, which yields a contradiction.
\end{proof}

\begin{lemma}
	The $E$-functions $E_0 (z), E_1 (z), \ldots, E_{r-1} (z)$ satisfy a system of first-order differential equation.
\end{lemma}

\begin{proof}
  Noting from the definition of the Euler operator that $\frac{\dd}{\dd z} = z^{-1} \vartheta$, we have
  \[
    \frac{\dd}{\dd z}
    \begin{pmatrix}
      E_0 (z) \\ \vdots \\ E_{r-1} (z)
    \end{pmatrix}
    = \frac{1}{z}
    \begin{pmatrix}
      0 & 1 & \cdots & 0 \\
      \vdots & \vdots & \ddots & \vdots \\
      0 & 0 & \cdots & 1 \\
      (rz)^r & 0 & \cdots & 0
    \end{pmatrix}
    \begin{pmatrix}
      E_0 (z) \\ \vdots \\ E_{r-1} (z)
    \end{pmatrix}
  \]
  Therefore, the collection of $E$-functions $Y = (E_0 (z), E_1 (z), \ldots, E_{r-1} (z))^{\top}$ satisfies the system of first-order differential equations $\frac{\dd}{\dd z} Y = A Y$, where $A \in \overline{\mathbb{Q}} (z)^{r \times r}$.
\end{proof}

\begin{proof}[Proof of \protect{\Cref{thm:Drx-Q-indep}}]
  We have already shown that the $E$-functions $E_0(z), E_1 (z), \ldots, E_{r-1} (z)$ satisfy the system of first-order differential equation $\frac{\dd}{\dd z} Y = A Y$ and are linearly independent over $\overline{\mathbb{Q}} (z)$. Moreover, the common denominator of the entries of $A$ is $T(z) = z$. Therefore, by \cref{thm:ref-SS}, for $x \in \overline{\bbQ} \setminus \{0\}$, the values $E_0 (x^{1/r}), \ldots, E_{r-1} (x^{1/r})$ are linearly independent over $\overline{\mathbb{Q}}$. Furthermore, for each $j = 0, 1, \dots, r-1$, we have
  \[
  	E_j(x^{1/r}) = \sum_{n=0}^\infty \frac{(rn)^j x^n}{(n!)^r} = r^j D_r(j; x),
  \]
  so that $D_r(0; x) \dots, D_r(r-1; x)$ are also linearly independent over $\overline{\bbQ}$.
\end{proof}

\subsection{Generalized Dobi\'{n}ski's formula}

We continue to consider the general object $D_r(n; x)$ introduced in \cref{def:Dr-x}. The common structure of the first part of \cref{Dobinski with r} and \cref{analogue of Dobinski with r} is captured by the following lemma.

\begin{lemma}\label{lem:DrN-rec}
	For any integers $r \ge 1, n \ge 0$ and $N \ge 1$, we define
	\[
		D_r^{(N)}(n; x) \coloneqq \sum_{k=0}^{N-1} \frac{k^n x^k}{(k!)^r}.
	\]
	Then, we have
	\[
		D_r^{(N)}(n+r; x) = x \sum_{k=0}^n \binom{n}{k} D_r^{(N)}(k; x) - \frac{N^n x^N}{((N-1)!)^r}.
	\]
\end{lemma}

\begin{proof}
	A direct calculation yields
	\[
		D_r^{(N)}(n+r; x) = \sum_{k=1}^{N-1} \frac{k^n x^k}{((k-1)!)^r} = \sum_{k=0}^{N-1} \frac{(k+1)^n x^{k+1}}{(k!)^r} - \frac{N^n x^N}{((N-1)!)^r}
	\]
	The claim then follows by applying the binomial theorem to $(k+1)^n$.
\end{proof}

From this, we simultaneously prove the following two theorems.

\begin{theorem}[A generalization of \cref{Dobinski with r}]
	For any integers $r \ge 1$ and $n \ge 0$, we have
	\[
		D_r(n; x) = \sum_{j=0}^{r-1} b_{r,j}(n; x) D_r(j; x),
	\]
	where the coefficients $b_{r,j}(n; x)$ are defined by $b_{r,j}(n; x) = \delta_{j,n}$ for $0 \le n \le r-1$ and the recurrence relation
	\begin{align}\label{def:brj-x-rec}
		b_{r,j}(n+r; x) = x \sum_{k=0}^n \binom{n}{k} b_{r,j}(k; x) \quad (n \ge 0).
	\end{align}
\end{theorem}

\begin{proof}
	Taking the limit $N \to \infty$ in \cref{lem:DrN-rec}, we obtain
	\begin{align}\label{eq:Dr-rec}
		D_r(n+r; x) = x \sum_{k=0}^n \binom{n}{k} D_r(k; x).
	\end{align}
	We prove the claim by induction on $n \ge 0$. It is clear from the definition of $(b_{r, j} (n;x))_{j=0}^{r-1}$ that the claim holds for $0 \le n \le r-1$. 
	
	Next, let $n \ge 0$, and assume that the claim holds for all $0 \le k \le n$. Then, by \eqref{eq:Dr-rec} and the induction hypothesis, we have
	\begin{align*}
	D_r (n+r;x) & = x\sum_{k=0}^n \binom{n}{k} D_r (k;x) = x\sum_{k=0}^n \binom{n}{k} \sum_{j=0}^{r-1} b_{r, j} (k;x) D_r (j;x) \\
		& = \sum_{j=0}^{r-1} \left( x\sum_{k=0}^n \binom{n}{k} b_{r, j} (k;x) \right) D_r (j;x) = \sum_{j=0}^{r-1} b_{r, j} (n+r;x) D_r (j;x).
	\end{align*}
	This completes the proof.
\end{proof}

Combining this with \cref{thm:Drx-Q-indep} specialized at $x=1$, we obtain \cref{Dobinski with r}.

\begin{remark}
	For $r=1$ and $j=0$, the polynomial $b_{1,0}(n; x)$ serves as the generating function of the \textit{Stirling number of the second kind}, that is,
	\[
		b_{1,0}(n; x) = \sum_{k=0}^n \stirlingII{n}{k} x^k,
	\]
	(see \cite[Proposition 2.6 (4)]{AIK2014} and OEIS~\cite[\href{https://oeis.org/A008277}{A008277}]{OEIS}).
\end{remark}

For a non-zero rational number $x$, we define 
\[
	D_{r,\calA}(n; x) \coloneqq (D_r^{(p)}(n; x) \bmod{p})_p \in \calA,
\]
as a generalization of \eqref{def:DrA}. 

\begin{theorem}[A generalization of \cref{analogue of Dobinski with r}]
	For any integers $r \ge 1$ and $n \ge 0$, we have
	\[
		D_{r,\calA} (n; x) = \sum_{j=0}^{r-1} b_{r,j}(n; x) D_{r, \calA}(j; x) + g_r(n; x),
	\]
	where the sequence $(g_r(n; x))_{n \ge 0}$ is given by the same recurrence relation
	\[
		g_r(n + r; x) = x\sum_{k = 0}^n\binom{n}{k}g_r(k; x)\quad (n\ge 1)
	\]
	as \eqref{def:brj-x-rec}, with the initial values $g_r(n; x) = (-1)^{r - 1} x \delta_{n,r}$ for $0\le n\le r$.
\end{theorem}

\begin{proof}
	For a large enough prime $p$, by \cref{lem:DrN-rec}, we have
	\begin{align}\label{eq:DrA-rec}
	\begin{split}
		D_r^{(p)}(n+r; x) &= x \sum_{k=0}^n \binom{n}{k} D_r^{(p)}(k; x) - \frac{p^n x^p}{((p-1)!)^r}\\
			&\equiv x \sum_{k=0}^n \binom{n}{k} D_r^{(p)}(k; x) + (-1)^{r-1} x \delta_{n,0} \pmod{p}.
	\end{split}
	\end{align}
	
	For $0 \le n < r$, the claim is clear. For $n=r$, since
	\[
		D_r^{(p)}(r; x) \equiv x D_r^{(p)}(0; x) + (-1)^{r-1}x \pmod{p}
	\]
	and $b_{r,j}(r; x) = x b_{r,j}(0; x) = x \delta_{j,0}$, we have 
	\[
		D_{r,\calA}(r; x) = \sum_{j=0}^{r-1} b_{r,j}(r;x) D_{r,\calA}(j, x) + g_r(r; x).
	\]
	
	Next, let $n \ge 1$ and assume that the claim holds for all integers $l$ with $0\le l\le n$. The recursion \eqref{eq:DrA-rec} and the definitions of $b_{r,j}$ and $g_r$ yield that
  \begin{align*}
    D_{r,\calA}(n+r;x)
    &= x \sum_{l=0}^{n}\binom{n}{l}D_{r,\calA}(l;x) \\
    &= x\sum_{l=0}^{n}\binom{n}{l}\left(\sum_{j=0}^{r-1} b_{r,j}(l;x)D_{r,\calA}(j;x)+g_r(l;x)\right) \\
    &= \sum_{j=0}^{r-1}\left(x\sum_{l=0}^{n}\binom{n}{l}b_{r,j}(l;x)\right)D_{r,\calA}(j;x)
      + x\sum_{l=0}^{n}\binom{n}{l}g_r(l;x).
  \end{align*}
  Thus, the claim holds for all $n\ge 0$.
\end{proof}


\section{Finite analogues of Euler's constant}\label{KMSw/parameters}

\subsection{Gregory polynomials}

As a generalization of the Gregory coefficients defined in \eqref{def:Gregory-coeff}, we define the \textit{Gregory polynomials} $G_n(x)$ by the generating function
\begin{align}\label{def:Gre-poly}
	\frac{t(1+t)^x}{\log(1+t)} = \sum_{n=0}^{\infty} G_n(x) t^n.
\end{align}
By definition, we have $G_n(0) = G_n$, and the first few examples are given by $G_0(x) = 1, G_1(x) = x+1/2$, and
\[
	G_2(x) = \frac{6x^2-1}{12}, \quad G_3(x) = \frac{4x^3-6x^2+1}{24}, \quad G_4(x) = \frac{30x^4 - 120x^3 + 120x^2 - 19}{720}.
\]
Independent studies of these polynomials were initiated by several authors around the 1920s. For example, Appell~\cite{Appell1926} investigated $G_n(x) - G_n$ under the name \textit{Fontana--Bessel polynomials}, while Jordan~\cite{Jordan1929} studied $G_n(x)$ under the name \textit{Bernoulli polynomials of the second kind}. For a detailed historical account, see Blagouchine~\cite[p.~16]{Blagouchine2018}.

Blagouchine~\cite{Blagouchine2018} summarizes several known results under the notation
\begin{align}\label{eq:Nnk-unit}
	N_{n,k}(x) \coloneqq \sum_{j=0}^{k-1} G_n(x+j).
\end{align}
Applying these results to the case $k=1$ yields results concerning $G_n(x)$. For example, in~\cite[(46)]{Blagouchine2018} we have
\begin{align}\label{eq:Gre-poly-exp}
\begin{split}
	G_n(x) &= \int_x^{x+1} \binom{u}{n} \dd u \quad (n \ge 0)\\
		&= \frac{(-1)^n}{n!} \sum_{j=1}^n \frac{(-1)^j}{j+1} \stirlingone{n}{j} ((x+1)^{j+1} - x^{j+1}) \quad (n \ge 1),
\end{split}
\end{align}
where $\stirlingone{\cdot}{\cdot}$ denotes the \textit{Stirling numbers of the first kind}, defined by
\[
	\frac{(\log(1+t))^j}{j!} = \sum_{n=j}^\infty (-1)^{n-j} \stirlingone{n}{j} \frac{t^n}{n!}.
\]
From this integral expression, noting the recurrence satisfied by binomial coefficients, we obtain
\begin{align}\label{eq:binom-rec}
	G_n(x) + G_{n-1}(x) = G_n(x+1).
\end{align}
Moreover, as recorded in~\cite[(61)]{Blagouchine2018}, for $x > -1$, we have the asymptotic formula
\begin{align}\label{eq:Nnk-asymp}
	G_n(x) \sim \frac{(-1)^{n+1}}{\pi n^{x+1} \log n} \left( \sin(\pi x) \Gamma(x+1) + \frac{\pi \cos(\pi x) \Gamma(x+1) + \sin(\pi x) \Gamma(x+1) \Psi(x+1)}{\log n} \right)
\end{align}
as $n \to \infty$, where $\Psi(x) \coloneqq \Gamma'(x)/\Gamma(x)$ is the \textit{digamma function}.

Blagouchine further presents the following identity concerning Euler's constant.

\begin{lemma}[{Blagouchine~\cite[(96), (101)]{Blagouchine2018}}]\label{lem:Bla-101}
	For an integer $k > 0$ and a real number $x > -1$, we have
	\[
		\gamma = \frac{1}{k} \sum_{n=1}^\infty \frac{(-1)^{n-1} N_{n,k}(x)}{n} - \frac{1}{k} \sum_{j=1}^k \log(x+j).
	\]
\end{lemma}

In particular, for $k=1$ and $x=0$, we obtain
\[
	\gamma = \sum_{n=1}^\infty \frac{(-1)^{n-1} G_n}{n},
\]
which is the classical formula due to Mascheroni. In what follows, we first generalize \cref{lem:Bla-101} to a Kluyver-type expression. For the results of Mascheroni and Kluyver, see~\cite{KMS2025}. Note that the case $x=0$ in the following proposition corresponds to Kluyver's formula. Furthermore, since \eqref{eq:Nnk-unit} holds, the general case for arbitrary $k$ can be reduced to the case $k=1$. Therefore, it suffices to consider the Gregory polynomials $G_n(x)$ for $k=1$.

\begin{proposition}\label{prpo:Mascheroni-x}
	For an integer $m \ge 0$ and a real number $x > -1$, we have
	\[
		\gamma = m! \sum_{n=1}^\infty \frac{(-1)^{n-1} G_{n}(x)}{(n)_{m+1}} + H_m - \log(x+m+1),
	\]
	where $H_m = \sum_{j=1}^m 1/j$ denotes the $m$-th harmonic number with $H_0 \coloneqq 0$, and $(n)_{m+1} = n(n+1) \cdots (n+m)$ is the rising factorial.
\end{proposition}

\begin{proof}
	The case $m=0$ is nothing but \cref{lem:Bla-101} for $k=1$. Hence, we assume $m > 0$ in what follows. First, note that
	\begin{align}\label{eq:telescope}
		\frac{1}{(n)_{m+1}} = \frac{1}{m} \left(\frac{1}{(n)_m} - \frac{1}{(n+1)_m} \right).
	\end{align}
	Using this, the infinite sum on the right-hand side can be rewritten as
	\begin{align*}
		\sum_{n=1}^\infty \frac{(-1)^{n-1} G_n(x)}{(n)_{m+1}} &= - \frac{1}{m} \frac{G_0(x)}{(1)_m} + \frac{1}{m} \sum_{n=1}^\infty \frac{(-1)^{n-1}}{(n)_m} (G_n(x) + G_{n-1}(x))\\
			&= -\frac{1}{m \cdot m!} + \frac{1}{m} \sum_{n=1}^\infty \frac{(-1)^{n-1} G_n(x+1)}{(n)_m},
	\end{align*}
	where the interchange of summations is justified by \eqref{eq:Nnk-asymp}, and the last equality uses~\eqref{eq:binom-rec}. Repeating this argument $m-1$ more times, we obtain
	\begin{align*}
		\sum_{n=1}^\infty \frac{(-1)^{n-1} G_n(x)}{(n)_{m+1}} = \cdots = -\frac{1}{m!} H_m + \frac{1}{m!} \sum_{n=1}^\infty \frac{(-1)^{n-1} G_n(x+m)}{n}.
	\end{align*}
	Finally, applying \cref{lem:Bla-101} with $k=1$ to the last sum gives the claimed formula.
\end{proof}

\subsection{Generalizations of Kaneko--Matsusaka--Seki's results}

As a generalization of Kaneko--Matsusaka--Seki's finite analogues of Euler's constant, the goal here is to introduce the finite analogues $\gamma_\calA^\mathrm{M}(x)$ and $\gamma_\calA^{\mathrm{K}, m}(x)$, and to show that they also differ from $\gamma_\calA^\mathrm{W}$ (defined in~\eqref{def:Wilson-gamma}) only by $\bbQ$-linear combinations of special values of $\log_\calA$ and $1$.

\begin{definition}
	For positive integers $m, k$ and any $x \in \bbQ$, we define
	\[
		\gamma_\mathcal{A}^{\mathrm{M}}(x) \coloneqq \left( \sum_{n=1}^{p-2} \frac{(-1)^{n-1} G_n(x)}{n} \bmod{p} \right)_p \in \mathcal{A}
	\]
	and 
	\[
		\gamma_\mathcal{A}^{\mathrm{K}, m}(x) \coloneqq \left(m!\sum_{n=1}^{p-m-1} \frac{(-1)^{n-1} G_n(x)}{(n)_{m+1}} \bmod{p}\right)_p + H_m - \ell_\calA(x+m+1) \in \mathcal{A}.
	\]
	In particular, the case $x=0$ corresponds the $\calA$-analogues of Euler's constants introduced in~\cite{KMS2025}.
\end{definition}

We investigate their properties in parallel with~\cite{KMS2025}. The introduction of the variable $x$ simplifies several arguments and provides benefits beyond a straightforward generalization.

\subsubsection{An analogue of $\gamma$ (Mascheroni), revisited}

For $x \in \bbQ^\times$, we define the \textit{Fermat quotient} by
\[
	q_p(x) \coloneqq \frac{x^{p-1} -1}{p}.
\]
We then set $\ell_\calA(x) \coloneqq (xq_p(x) \bmod{p})_p = x \log_\calA(x) \in \calA$. It is worth noting that $\ell_\calA(0) = 0$ holds.

\begin{theorem}[{A generalization of \cite[Theorem 2.4]{KMS2025}}]\label{thm:Mascheroni}
	For any $x \in \bbQ$, we have
	\[
		\gamma_\mathcal{A}^\mathrm{M}(x) = \gamma_\mathcal{A}^\mathrm{W} + \ell_\mathcal{A}(x+2) - \ell_\mathcal{A}(x+1) +\delta_{-1}(x) - 1,
	\]
	where $\delta_{-1}(x)$ denotes the indicator function of $\{-1\}$. In particular, we have
	\[
		\gamma_\calA^\mathrm{M}(-1) = \gamma_\calA^\mathrm{W}.
	\]
\end{theorem}

To prove this theorem, we first prepare the following lemma.

\begin{lemma}\label{lem:binom-cong}
	Let $x \neq -1$ be a rational number. For a sufficiently large prime $p$, we have 
	\[
		\binom{x}{p-1} \equiv \delta_{-1}(x) \pmod{p}.
	\]
\end{lemma}

\begin{proof}
	If $x=-1$, then we have $\binom{-1}{p-1} \equiv 1 \pmod{p}$ for odd prime $p$. If $x \neq -1$, then for sufficiently large primes $p$ we have $x \not\equiv -1 \pmod{p}$, which implies that $\binom{x}{p-1} \equiv 0 \pmod{p}$.
\end{proof}

%

\begin{proof}[Proof of \cref{thm:Mascheroni}]
	Multiplying both sides of the generating function~\eqref{def:Gre-poly} by $\log(1+t)$ and comparing the coefficients in the power series expansion at $t=0$, we obtain
	\[
		\sum_{n=1}^{k-1} \frac{(-1)^{n-1} G_n(x)}{k-n} = (-1)^k \binom{x}{k-1} + \frac{1}{k}.
	\]
	Now, for a sufficiently large prime $p$, setting $k=p$, we have
	\begin{align}\label{eq:Masche-L}
		\sum_{n=1}^{p-2} \frac{(-1)^{n-1} G_n(x)}{n} &\equiv -\sum_{n=1}^{p-1} \frac{(-1)^{n-1} G_n(x)}{p-n} - G_{p-1}(x) \nonumber \\
			&\equiv \binom{x}{p-1} - \frac{1}{p} - G_{p-1}(x) \nonumber \\
			&\equiv - \frac{1}{p} - G_{p-1}(x) + \delta_{-1}(x) \pmod{p},
	\end{align}
	where we applied \cref{lem:binom-cong}.
	
	On the other hand, by \eqref{eq:Gre-poly-exp}, we have
	\[
		G_{p-1}(x) = \frac{1}{(p-1)!} \sum_{j=1}^{p-2} \frac{(-1)^j}{j+1} \stirlingone{p-1}{j} ((x+1)^{j+1} - x^{j+1}) + \frac{(x+1)^p - x^p}{p!}.
	\]
	As in the proof of~\cite[Theorem 2.4]{KMS2025}, by applying $\stirlingone{p-1}{j} \equiv 1 \pmod{p}$ for $j \ge 1$, the identity
	\[
		\frac{(x+1)^p - x^p}{p!} = \frac{1}{(p-1)!} \left((x+1) q_p(x+1) - xq_p(x) + \frac{1}{p} \right),
	\]
	and Wilson's theorem, it follows that
	\[
		G_{p-1}(x) - \frac{1}{p!} \equiv -\sum_{j=1}^{p-2} \frac{(-1)^j}{j+1} ((x+1)^{j+1} - x^{j+1}) - (x+1)q_p(x+1) + xq_p(x) \pmod{p}.
	\]
	Furthermore, applying Eisenstein's congruence~\cite{Eisenstein1850}
	\[
		\sum_{m=1}^{p-1} \frac{(-1)^{m-1} x^m}{m} \equiv (x+1) q_p(x+1) - xq_p(x) \pmod{p},
	\]
	we obtain
	\begin{align}\label{eq:Masche-R}
		G_{p-1}(x) - \frac{1}{p!} \equiv - (x+2) q_p(x+2) + (x+1) q_p(x+1) +1 \pmod{p}.
	\end{align}
	
	Combining \eqref{eq:Masche-L} and \eqref{eq:Masche-R}, we finally obtain
	\begin{align*}
		\sum_{n=1}^{p-2} \frac{(-1)^{n-1} G_n(x)}{n} &\equiv -G_{p-1}(x) + \frac{1}{p!} - \frac{(p-1)!+1}{p (p-1)!} + \delta_{-1}(x)\\
			&\equiv (x+2) q_p(x+2) - (x+1) q_p(x+1) + \frac{(p-1)! +1}{p} + \delta_{-1}(x)-1 \pmod{p},
	\end{align*}
	which implies the desired result.
\end{proof}

\subsubsection{An interlude (Gregory), revisited}

As a generalization of $G_\calA(k)$ introduced in~\cite{KMS2025}, we introduce the following.

\begin{definition}
	For $k \ge 2$ and $x \in \bbQ$, we define $G_\calA(k; x) \in \calA$ by
	\[
		G_\calA(k; x) \coloneqq (G_{p-k}(x) \bmod{p})_p.
	\]
\end{definition}

The following lemma simplifies the proof of \cite[Theorem 3.2]{KMS2025}.

\begin{lemma}\label{lem:Gregory-shift}
	For any integers $n, N \ge 0$, we have
	\[
		G_n(x) = (-1)^N \sum_{j=0}^N (-1)^j \binom{N}{j} G_{n+N}(x+j).
	\]
\end{lemma}

\begin{proof}
	By induction on $N$. Indeed, the case $N=0$ is clear. Assuming the statement holds for $N$, the case $N+1$ follows by applying \eqref{eq:binom-rec} and simplifying.
\end{proof}

\begin{theorem}[{A generalization of \cite[Theorem 3.2]{KMS2025}}]\label{thm:Interlude}
	For $k \ge 2$ and a rational number $x$, we have
	\[
		G_\calA(k; x) = (-1)^{k-1} \sum_{j=0}^k (-1)^j \binom{k}{j} \ell_\calA(x+j+1)
	\]
\end{theorem}

\begin{proof}
	It immediately follows from applying \cref{lem:Gregory-shift} for $n = p-k$ and $N = k-1$, together with \eqref{eq:Masche-R} that
  \[
    G_\calA(k; x) = (-1)^{k-1} \sum_{j=0}^{k-1} (-1)^j \binom{k-1}{j} (\ell_\calA(x+j+1) - \ell_\calA(x+j+2)).
  \]
  The claim follows by a simple change of variables and straightforward simplification.
\end{proof}

\subsubsection{Yet more variations (Kluyver), revisited}

Finally, we consider the Kluyver-type $\gamma_\calA^{\mathrm{K},m}(x)$. As in the proof of \cref{prpo:Mascheroni-x}, the proof of the following theorem reduces to \cref{thm:Mascheroni}.

\begin{theorem}[{A generalization of~\cite[Theorem 4.1]{KMS2025}}]
	For each $m \ge 1$ and a rational number $x$, we have
	\[
		\gamma_\calA^{\mathrm{K},m}(x) = \gamma_\calA^\mathrm{W} + \delta_{-1}(x+m) -1 + (H_m-1) \ell_\calA(x+m+1) + \sum_{j=0}^{m-1} (-1)^{m-j} \binom{m}{j} \frac{\ell_\calA(x+j+1)}{m-j}.
	\]
\end{theorem}

\begin{proof}
	Let $p$ be a sufficiently large prime. By \eqref{eq:telescope}, we have
	\begin{align*}
		m! \sum_{n=1}^{p-m-1} \frac{(-1)^{n-1} G_n(x)}{(n)_{m+1}} &= (m-1)! \sum_{n=1}^{p-m-1} \left(\frac{1}{(n)_m} - \frac{1}{(n+1)_m}\right) (-1)^{n-1} G_n(x)\\
			&= (m-1)! \left(\sum_{n=1}^{p-m} \frac{(-1)^{n-1}}{(n)_m} (G_n(x) + G_{n-1}(x)) - \frac{(-1)^m G_{p-m}(x)}{(p-m)_m} - \frac{1}{m!}\right)\\
			&= (m-1)! \sum_{n=1}^{p-m} \frac{(-1)^{n-1} G_n(x+1)}{(n)_m} + \frac{(-1)^{m-1} (m-1)!G_{p-m}(x)}{(p-m)_m} - \frac{1}{m},
	\end{align*}
	where we applied \eqref{eq:binom-rec}.
	Iterating this argument yields the following:
	\begin{align*}
		m! \sum_{n=1}^{p-m-1} &\frac{(-1)^{n-1} G_n(x)}{(n)_{m+1}}\\ 
		&\equiv \cdots \equiv \sum_{n=1}^{p-1} \frac{(-1)^{n-1} G_n(x+m)}{n} + \sum_{j=1}^m \frac{(-1)^{j-1} (j-1)!G_{p-j}(x+m-j)}{(p-j)_j} - H_m\\
			&\equiv \sum_{n=1}^{p-2} \frac{(-1)^{n-1} G_n(x+m)}{n} + G_{p-2}(x+m-1) - \sum_{j=2}^m \frac{G_{p-j}(x+m-j)}{j} - H_m \pmod{p},
	\end{align*}
	where we use $G_{p-1}(x+m) - G_{p-1}(x+m-1) = G_{p-2}(x+m-1)$ and $(p-j)_j \equiv (-1)^j j! \pmod{p}$. Therefore, by \cref{thm:Mascheroni} and \cref{thm:Interlude}, we have
	\begin{align*}
		\gamma_\calA^{\mathrm{K},m}(x) &= \gamma_\calA^\mathrm{M}(x+m) + G_\calA(2; x+m-1) - \sum_{j=2}^m \frac{G_\calA(j; x+m-j)}{j} - \ell_\calA(x+m+1)\\
			&= \gamma_\calA^\mathrm{W} + \delta_{-1}(x+m) - 1 - \ell_\calA(x+m) - \sum_{j=2}^m \frac{(-1)^{j-1}}{j} \sum_{l=0}^j (-1)^l \binom{j}{l} \ell_\calA(x+m-j+l+1)
	\end{align*}
	A straightforward simplification of this expansion yields the desired result.
\end{proof}

\begin{remark}
	In the theorems of this section, expressions of the form $\ell_\calA(x) - \ell_\calA(x-1)$ appear frequently. This coincides with the $\calA$-analogue of another type of logarithmic function introduced by Kaneko--Zagier~\cite[Example 3]{KanekoZagier}, namely
	\[
		L_1(x) = \left(-\sum_{n=1}^{p-1} \frac{(1-x)^n}{n} \bmod{p}\right)_p = x \log_\calA(x) - (x-1) \log_\calA(x-1).
	\]
\end{remark}

\bibliographystyle{plainurl}
\bibliography{references}

\end{document}